\documentclass[a4paper,11pt]{article}

\input{styleart.sty}
\usepackage{srcltx}

\newcommand{\s}{\sigma}

\title{On groups and fields interpretable in torsion-free hyperbolic groups\thanks{Research supported by SFB 878}}
\date{\today}
\author{Chlo\'e Perin \\University of Strasbourg\and Anand Pillay\thanks{Supported by EPSRC grant EP/I0002294/1}\\University of Leeds\and 
Rizos Sklinos\\Hebrew University of Jerusalem\and Katrin Tent\\Universit\"at M\"unster} 

\begin{document}

\maketitle

\begin{abstract} We prove that the generic type of a non-cyclic torsion-free hyperbolic group $G$ is 
foreign to any interpretable abelian group, hence also to any interpretable field. 
This result depends, among other things, on the definable simplicity of a non-cyclic torsion-free hyperbolic group, 
and we take the opportunity to give a proof of the latter using Sela's description of imaginaries in torsion-free hyperbolic groups. We also use the 
description of imaginaries to prove that if $\F$ is a free group of rank $>2$ then no orbit of a finite tuple from $\F$ under $Aut(\F)$ is definable.
\end{abstract}

\section{Introduction} 

This paper concerns the first order theories of torsion-free hyperbolic groups. There is an increasing 
model theoretic interest in the subject motivated by the positive solution to Tarski's problem (i.e. is the theory of non abelian free groups complete?) 
by Sela and Kharlampovich-Myasnikov. 
Subsequently Sela proved the stability of
all non-cyclic torsion-free hyperbolic groups \cite{SelaStability}. These are in fact remarkable examples of ``new
stable groups'', given by nature.

If we fix a torsion-free hyperbolic group $G$, then understanding the category of 
definable/interpretable sets in models of $Th(G)$, 
informed by stability-theoretic tools and notions, is a challenge. 
Sela's work on imaginaries in torsion-free hyperbolic groups \cite{SelaIm} is 
part of this endeavour and will be used in the current paper. 

Our paper contributes to the following conjecture:

\begin{conjecture}\label{Field}
Let $G$ be a torsion-free hyperbolic group. Then no infinite field is interpretable in (any model of) $Th(G)$.
\end{conjecture}

In general, also the nature and complexity of interpretable groups in a theory is important and we make another 
conjecture which will only indirectly be touched on in the current paper:

\begin{conjecture}  
Let $G$ be a torsion-free hyperbolic group. Any group interpretable in a model $\mathcal{M}$ of $Th(G)$ is 
definably isomorphic to a definable subgroup of $\mathcal{M}\times\ldots\times \mathcal{M}$.
\end{conjecture}

Note that this is basically the situation in $1$-based groups (see \cite{1BEPP}).

Let $G$ be a non-cyclic torsion-free hyperbolic group. Then the first order theory of $G$ 
is connected and hence has a unique generic type, which we call $p^{G}_{0}$ (see section \ref{Ima}). 
In the special case of non abelian free groups a considerable amount of information has been obtained 
around the complexity of $p^{\F_n}_{0}$ (by all the authors). 
For example $p^{\F_n}_{0}$ has infinite weight (\cite{PillayGenericity}, \cite{SklinosGenericType}), 
and witnesses the fact that the free group is $n$-ample for all $n$ (\cite{AmpleSk}). 
As a matter of fact, in \cite{AmpleOK} it was proved that the theory of any (non-cyclic) torsion-free 
hyperbolic group is $n$-ample for all $n$. 
The ampleness result is consistent with the existence of an infinite interpretable field which interacts with $p^G_{0}$. 
Corollary \ref{Skew} rules this out. Thus, the results of the current paper provide a partial solution to Conjecture \ref{Field}, 
yielding additional information on $p^G_{0}$: no interpretable abelian group (hence also interpretable field) can interact with $p^G_{0}$. 
See Section \ref{Generic} for the stability-theoretic definitions.

\begin{thmIntro}\label{For}
Let $G$ be a non-cyclic torsion-free hyperbolic group. Then $p^G_{0}$ is foreign to any interpretable abelian group.
\end{thmIntro}

\begin{corIntro}\label{Skew} 
Let $G$ be a non-cyclic torsion-free hyperbolic group. Then $p^G_{0}$ is foreign to any interpretable field or skew field.
\end{corIntro}

We will need to know that non-cyclic torsion-free hyperbolic groups are definably simple (no definable proper nontrivial normal subgroup). 
The stronger result
that the only proper definable subgroups of torsion-free hyperbolic groups are cyclic, 
has been stated in several places, such as \cite{MalcevKM}. We take the opportunity here to give an independent proof in Section \ref{NDR}, 
using Sela's description of imaginaries \cite{SelaIm} as a ``black box'':

\begin{thmIntro} \label{Mal}
The only definable proper subgroups of a torsion-free hyperbolic group are cyclic.
\end{thmIntro}

In the same section we also prove:

\begin{thmIntro} \label{Orb3}
Let $\F$ be a free group of rank at least $3$. Then no orbit of a finite (nontrivial) tuple under $Aut(\F)$ is definable.
\end{thmIntro}

Let us stress that by \emph{definable} we mean definable possibly \emph{with parameters}.

The paper is organised as follows. In the following section we give some model theoretic background 
around torsion-free hyperbolic groups. We give special emphasis to the notion of imaginaries  
and we give a precise account of elimination of imaginaries in model theoretic terminology as this is crucial
for questions regarding interpretability. 

In Section \ref{NDR}, we give certain non-definability results. In particular we use Sela's elimination of imaginaries result 
to prove Theorem \ref{Mal} and Theorem \ref{Orb3}.

In Section \ref{AbIntGr} we prove a result that forbids abelian interpretable groups 
in non abelian free groups to ``gain'' an element in higher rank free groups. This result is an elaboration  on material in 
Chapter 8 of the third author's Ph.D. thesis \cite{ThesisSklinos}, but now we have to consider not only real tuples but imaginaries. 
We also prove an analogous result for torsion-free hyperbolic groups.

Finally in Section \ref{Generic}, we bring everything together to prove Theorem \ref{For}.

\section{Some model theory of torsion-free hyperbolic groups}\label{Ima}

We start our discussion with the free group case. Our notation is fairly standard. By $\F_{n}$ we denote the free group on $n$ generators 
and we usually denote a basis of $\F_n$ by $e_{1},..., e_{n}$. By $T_{fg}$ we denote the common theory of non abelian free groups. 
We also note that the natural embedding of $\F_{n}$ in $\F_{m}$ for $2\leq n < m$ is elementary (as proved by Sela and Kharlampovich-Myasnikov).

In \cite{PoizatGenericAndRegular} Poizat proved that $\F_{\omega}$ is connected (thus $T_{fg}$ is connected). 
Moreover the following theorem has been proved by the 
second named author in \cite{PillayForking}. 

\begin{theorem}
Let $\F_{\omega}:=\langle e_1,\ldots,e_n,\ldots\rangle$. Then $(e_i)_{i<\omega}$ is a Morley sequence in $p^{\F_n}_0$. 
In particular $tp(e_{n+1}/\F_n)$ is generic.
\end{theorem}

Now, let $G$ be a torsion-free hyperbolic group. Sela assigns to such a group its elementary core $EC(G)$, which is
an  elementary subgroup of $G$ provided $G$ is not elementarily equivalent to a
free group (for a definition and further properties see
\cite{Sel7}). Note that if $G$ is elementarily equivalent  to a free group,
$EC(G)$ is the trivial group. The elementary core of $EC(G)*\mathbb{Z}$ is again $EC(G)$, so the latter is an elementary subgroup of the former. 
This observation led Ould Houcine to the following
result \cite{OuldHoucineHomogeneity}

\begin{theorem} \label{HypGen}
Let $G$ be a torsion-free hyperbolic group not elementarily equivalent to a free group. Then $G$ is connected. 
Moreover if $H:=EC(G)*\langle e\rangle$, then $tp^{H}(e/EC(G))$ is generic. 
\end{theorem}

In \cite{AmpleOK} the following useful result was proved.
 
\begin{proposition}\label{HypElem}\label{HypChain}
Let $G$ be a torsion-free hyperbolic group not elementarily equivalent to a free group. Suppose $K$ is a free factor 
of a free group $\F$. Then $EC(G)*K$ is an elementary subgroup of $G*\F$. In particular since $EC(EC(G))=EC(G)$ 
we have an elementary chain $$EC(G)\prec EC(G)*\mathbb Z\prec\ldots \prec EC(G)*\F_n\prec\ldots.$$
\end{proposition}

\subsection{Imaginaries in torsion-free hyperbolic groups}

We first give a quick overview of the model theoretic notion of imaginaries as well as various notions of elimination of imaginaries. 
We then specialize to torsion-free hyperbolic groups and give Sela's result.

Recall that $\mathcal{M}^{eq}$ is constructed from $\mathcal{M}$ 
by adding a new sort for each $\emptyset$-definable equivalence relation, $E(\bar{x},\bar{y})$, together 
with a class function $f_E:M^n\rightarrow M_E$, where $M_E$ (the domain of the new sort corresponding to $E$) 
is the set of all $E$-equivalence classes. The elements in these new sorts are called {\em imaginaries}. 
Note that any automorphism of $\cal M$ has a canonical extension to ${\cal M}^{eq}$.

We say that $\mathcal{M}$ eliminates imaginaries if it has a saturated elementary extension $\mathbb{M}$ 
in which for any element $\mathfrak{e}$ of $\mathbb{M}^{eq}$, 
there is a finite tuple $\bar{b}\in\mathbb{M}$ such that $\mathfrak{e}\in dcl^{eq}(\bar{b})$ and $\bar{b}\in dcl^{eq}(\mathfrak{e})$. 

We say that $\mathcal{M}$ weakly eliminates imaginaries if it has a saturated elementary extension $\mathbb{M}$ 
in which for any element $\mathfrak{e}$ of $\mathbb{M}^{eq}$, 
there is a finite tuple $\bar{b}\in\mathbb{M}$ such that $\mathfrak{e}\in dcl^{eq}(\bar{b})$ 
and $\bar{b}\in acl^{eq}(\mathfrak{e})$.

We now specialize to torsion-free hyperbolic groups.

\begin{definition} \label{BasicSorts}
Let $G$ be a torsion-free hyperbolic group. The following equivalence relations in $G$ are called basic.
\begin{itemize}
 \item[(i)] $E_1(a,b)$ if and only if there is $g\in G$ such that $a^g=b$. (conjugation)
  \item[$(ii)_m$] $E_{2_m}((a_1,b_1),(a_2,b_2))$ if and only if either $b_1=b_2=1$ or $b_1\neq 1$ and $C_{G}(b_1)=C_{G}(b_2)=\langle b \rangle$ and
 $a_1^{-1}a_2\in\langle b^m \rangle$. ($m$-left-coset)
  \item[$(iii)_m$] $E_{3_m}((a_1,b_1),(a_2,b_2))$ if and only if either $b_1=b_2=1$ or $b_1\neq 1$ and $C_{G}(b_1)=C_{G}(b_2)=\langle b \rangle$ and
 $a_1a_2^{-1}\in\langle b^m \rangle$. ($m$-right-coset)
  \item[$(iv)_{m,n}$] $E_{4_{m,n}}((a_1,b_1,c_1),(a_2,b_2,c_2))$ if and only if either $a_1=a_2=1$ or $c_1=c_2=1$ or 
 $a_1,c_1\neq 1$ and $C_{G}(a_1)=C_{G}(a_2)=\langle a \rangle$ and $C_{G}(c_1)=C_{G}(c_2)=\langle c \rangle$
   and there is $\gamma\in \langle a^m \rangle$ and $\epsilon\in \langle c^n \rangle$ such that $\gamma b_1 \epsilon=b_2$. ($m,n$-double-coset)
 \end{itemize}
 
 \end{definition}
 
 The following is Theorem 4.4 in \cite{SelaIm}:
 
 \begin{theorem}\label{Elim}
 Let $G$ be a torsion-free hyperbolic group. Let $E(\bar{x},\bar{y})$ be a definable equivalence relation in $G$, with $\abs{\bar{x}}=m$.
 Then there exist $k,l<\omega$ and a definable relation
 $$R_E \subseteq G^m \times G^k \times S_1(G) \times \ldots \times S_l(G)$$
 such that:
 \begin{itemize}
  \item[(i)] each $S_i(G)$ is one of the basic sorts;
  \item[(ii)] there is some $s\in\N$ such that for each $\bar{a}\in G^m$, $\abs{R_E(\bar{a},G^{eq})}\leq s$;
  \item[(iii)] $\forall\bar{z}(R_E(\bar{a},\bar{z})\leftrightarrow R_E(\bar{b},\bar{z}))$ if and only if $E(\bar{a},\bar{b})$.
 \end{itemize}
 \end{theorem}

In the case of $\emptyset$-definable equivalence relations a slight variation of the above theorem is true 
(see Theorem 4.4 and Proposition 4.5 in \cite{SelaIm}):

 \begin{theorem}\label{0Elim}
 Let $G$ be a torsion-free hyperbolic group. Let $E(\bar{x},\bar{y})$ be a $\emptyset$-definable equivalence relation in $G$, with $\abs{\bar{x}}=m$.
 Then there exist $k,l<\omega$, finitely many ``exceptional'' $\emptyset$-definable equivalence classes, defined by $\phi_i(\bar{x})$, $i\leq n$, 
 and a $\emptyset$-definable relation
 $$R_E \subseteq G^m \times G^k \times S_1(G) \times \ldots \times S_l(G)$$
 such that:
 \begin{itemize}
  \item[(i)] each $S_i(G)$ is one of the basic sorts;
  \item[(ii)] there is some $s\in\N$ such that for each $\bar{a}\in G^m$, $\abs{R_E(\bar{a},G^{eq})}\leq s$;
  \item[(iii)] $G^{eq}\models\forall\bar{x}(\phi_1(\bar{x})\lor\phi_2(\bar{x})\lor\ldots\lor\phi_n(\bar{x})
  \rightarrow R_E(\bar{x},1,1,\ldots,(1,1,1))\land \abs{R_E(\bar{x},\bar{z})}=1)$, 
  i.e. $R_E$ assigns to each tuple in any of the ``exceptional'' classes a single tuple consisting of identity elements in the corresponding sorts; 
  \item[(iv)] for any two tuples $\bar{a},\bar{b}$ not in the union of the ``exceptional'' classes we have   
  $G^{eq}\models\forall\bar{z}(R_E(\bar{a},\bar{z})\leftrightarrow R_E(\bar{b},\bar{z}))$ if and only if $E(\bar{a},\bar{b})$.
 \end{itemize}
 \end{theorem}

Let $G$ be a group elementarily equivalent to a torsion-free hyperbolic group, 
by $G^{we}$ we denote the expansion of $G$ by the above basic sorts, i.e. 
$G^{we}=(G,S_1(G),\{S_{2_m}(G)\}_{m<\omega},$ $\{S_{3_m}(G)\}_{m<\omega},\{S_{4_{m,n}}(G)\}_{m,n<\omega})$. 
The above theorem easily implies that $G^{we}$ weakly eliminates imaginaries.

\begin{corollary}
Let $G$ be a torsion-free hyperbolic group. Then $G^{we}$ weakly eliminates imaginaries.
\end{corollary}
\begin{proof}
We work in a saturated extension $\mathbb{G}$ of $G$. Let $\mathfrak{e}$ be an element of $\mathbb{G}^{eq}$. 
Then there is a $\emptyset$-definable equivalence relation $E(\bar{x},\bar{y})$ in $G$ such that 
$\mathfrak{e}=[\bar{a}]_E$ for some element $\bar{a}\in\mathbb{G}$. By Theorem \ref{0Elim}, $E$ is assigned a $\emptyset$-definable 
relation $R_E$ such that $R_E(\bar{a},\bar{z})$ has finitely many solutions in $\mathbb{G}^{we}$. 

First assume that 
$\bar{a}$ belongs to an ``exceptional'' class, then we claim that $\mathfrak{e}\in dcl^{eq}(1)$, 
indeed since each exceptional class is $\emptyset$-definable any automorphism will fix $\mathfrak{e}$. 
Now suppose that $\bar{a}$ does not belong to an exceptional class and $(\bar{a}_1\bar{a}_2\ldots\bar{a}_k)$ is the concatenation of 
the solutions of $R_E(\bar{a},\bar{z})$ in $G^{eq}$. We claim that 
$\mathfrak{e}\in dcl^{eq}(\bar{a}_1\bar{a}_2\ldots\bar{a}_k)$ and $\bar{a}_1\bar{a}_2\ldots\bar{a}_k\in acl^{eq}(\mathfrak{e})$. 
The formula $\exists\bar{y}(f_E(\bar{y})=x\land R_E(\bar{y},\bar{a}_1)\land\ldots\land R_E(\bar{y},\bar{a}_k) \land 
\forall \bar{w} ( R_E(\bar{y}, \bar{w}) \to \bigvee^{k}_{i=1} \bar{w} = \bar{a}_i))$ 
defines $\mathfrak{e}$ in $\mathbb{G}^{eq}$. Now consider the formula $\exists \bar{y}(R_E(\bar{y},\bar{x}_1)\land\ldots\land R_E(\bar{y},\bar{x}_k)\land 
f_E(\bar{y})=\mathfrak{e})$, the solution set of this formula is $\{ \bar{a}_1, \bar{a}_2, \ldots, \bar{a}_k \}^k$, which is finite.
\end{proof}
 
\section{Some (non) definability results}\label{NDR}
We first prove that the only definable proper subgroups of a torsion-free hyperbolic group 
are cyclic. This implies, in the case where $G$ is non-cyclic, that $Th(G)$ is definably simple. 

 
Let $G$ be a non-cyclic torsion-free hyperbolic group and $H:=G*\langle e\rangle$. 
If $a$ is an element in $G$, we denote by $f_a$ the automorphism of $H$ which is the identity on $G$ and satisfies $f_a(e)=ea$.

\begin{lemma}\label{MalOrb}
Let  $\beta \in H^{eq}\setminus G^{eq}$ be an element from a basic
sort. Suppose $\beta$ is not contained in the following list:

\begin{enumerate}
\item $\beta=[b]_{E_1}$, where $b=b_1e^{i_1}b_2\ldots b_m e^{i_m}$ with 
$m>1$, $b_i\in G\setminus\{1\}$ for all $i\leq m$ and $b_i\in C(a)$ for some $i\leq m$; or

\item $\beta=[(c,b)]_E$ for $E=E_{2_p}$ or $E=E_{3_p}$, or $\beta=[(b,c,d)]_{E_{4_{p,q}}}$, 
where $b=b_1e^{i_1}b_2e^{i_2}\ldots b_m$ $e^{i_m}b_{m+1}$ with $m>1$, $b_i\in G$ (if $1<i<m+1$, then $b_i$ is non trivial) and $b_2\in C(a)$; or

\item $\beta=[(c,b)]_E$ for $E=E_{2_p}$ or $E=E_{3_p}$, or $\beta=[(b,c,d)]_{E_{4_{p,q}}}$, where $b,d\in G$ and 
$c=c_1e^{i_1}c_2\ldots c_m e^{i_m}c_{m+1}$ with $c_i\in G$ (if $1<i<m+1$, then $c_i$ is non trivial). 
Moreover, either some $c_i$, for $1<i<m+1$, is in $C(a)$ or $b$ and/or $d$ belong to the centralizer of some conjugate of $a$ 
(by an element of $G$).
 
\end{enumerate}
Then $\beta$ has infinite orbit under $\langle f_a^i : i\in\Z\rangle$. 
\end{lemma}

\begin{proof}
Assume that $\beta$ has finite orbit. We take cases according to the list of basic sorts.
\begin{itemize}
 \item Let $\beta=[b]_{E_1}$. Let $b_1e^{i_1}b_2\ldots b_m e^{i_m}$ be the normal form of $b$ 
with respect to the decomposition $G*\langle e\rangle$. Suppose for some $j\leq m$ we have that $\abs{i_j}>1$. Then 
$ea^ie$ or $e^{-1}a^{-i}e^{-1}$ will appear in the normal form of $f^i(b)$ and the result follows easily. Suppose, for some $j$ we 
have that $i_{j (mod \ m)}=i_{j+1 (mod \ m)}$. Then $ea^ib_{j+1 (mod \ m)} e$ or 
$e^{-1}b_{j+1 (mod \ m)} a^{-i}e^{-1}$ will appear in the normal form of $f^i(b)$ so there must be $k$ so that $b_k = a^ib_{j+1 (mod \ m)}$ or $b_k = b_{j+1 (mod \ m)} a^{-i}$ for infinitely many values of $i$, which is impossible. If $m=1$, then is not hard to see that $[b_1e^{\pm 1}]_{E_1}$ 
has infinite orbit under $\langle f^i | i<\omega\rangle$. Lastly, suppose for some $j$,  $i_{j(mod \ m)}=1$ and $i_{j+1 (mod \ m)}=-1$ 
but $b_{j+1 (mod \ m)}\not\in C(a)$. Then $e a^ib_{j+1 (mod \ m)} a^{-i} e^{-1}$ will appear in the normal form of 
$f^i(b)$ and since $[a,b_{j+1 (mod \ m)}]\neq 1$ the result follows.
 \item Let $\beta=[(b,c,d)]_{E_{4_{p,q}}}$. We take two subcases: 
   \begin{itemize}
    \item[(i)] Suppose that $b\in H\setminus G$. 
    Since $H$ is a torsion-free hyperbolic group we have that there is a $\gamma\in H\setminus G$ such that  
    $f^i(b)=\gamma$ for infinitely many $i$'s. Now, let $b_1e^{i_1}b_2\ldots b_m e^{i_m}b_{m+1}$ be the normal form of 
    $b$ with respect to the decomposition $G*\langle e\rangle$. It follows as in the case above that $\abs{i_j}=1$ for all $j\leq m$, 
    and $i_j\neq i_{j+1}$ for all $j<m$. Moreover, $m>1$ and if for some $j<m$, $i_j=1$ and $i_{j+1}=-1$, 
    then $b_{j+1}\in C(a)$. What remains to be shown is that $i_1\neq -1$, which is trivial to check.
    \item[(ii)] Suppose $b,d\in G$ and let $c_1e^{i_1}c_2\ldots c_m$ $ e^{i_m}c_{m+1}$ 
    be the normal form of $c$ with respect to the decomposition $G*\langle e\rangle$. Suppose that $\abs{i_j}>1$ for some $j\leq m$, 
    then $f^i(c)$ will have a normal form containing $ea^ie$ or $e^{-1}a^{-i}e^{-1}$, thus 
    $\epsilon f^i(c)\delta\neq f^j(c)$ for any $\epsilon,\delta\in G$ and the result follows. 
    Similarly one can prove that $i_j\neq i_{j+1}$, for all $j<m$ and if for some $j<m$, $i_j=1$ and $i_{j+1}=-1$ then $c_{j+1}\in C(a)$.  
    We now take cases that depend on the value of $m$. If $m>2$, it is clear that either $c_2$ or $c_3$ will be in $C(a)$. 
    Suppose $m=1$ and $c=c_1ec_2$ (or $c=c_1e^{-1}c_2$). Then 
    $f^i(c)=c_1ea^ic_2$ (or $f^i(c)=c_1a^{-i}e^{-1}c_2$) and $\epsilon f^i(c)\delta=f^j(c)$ only if 
    $\delta\in C(c_2^{-1}ac_2)$ (only if $\epsilon \in C(c_1ac_1^{-1}$)). Lastly suppose 
    $m=2$, in this case we can only have $i_1=-1$ and $i_2=1$ (otherwise $c_2\in C(a)$), 
    thus we have $c=c_1e^{-1}c_2ec_3$ and 
    $\epsilon f^i(c)\delta=f^j(c)$ only if $\epsilon\in C(c_1ac_1^{-1})$ and $\delta\in C(c_3^{-1}ac_3)$.
   \end{itemize}
\item Let $\beta=[(c,b)]_E$ for $E=E_{2_p}$ or $E=E_{3_p}$. The proof is similar to the above and is left to the reader.  
\end{itemize}

\end{proof}

This immediately yields the following:

\begin{corollary}\label{corMalOrb}\label{InfOrb} 
Let $a,c\in G$ with $C(a) \neq C(c)$. 
Then for any element $\beta\in H^{we}\setminus G^{we}$, $\beta$ has infinite orbit 
under $\langle f^i_a, f^i_c | i<\omega\rangle$. 
In particular, if $G$ is a non-trivial free product and $\beta\in H^{we}$, 
then $Aut(H).\beta$ is infinite.
\end{corollary}

We note that the particular case of Corollary \ref{InfOrb} fails in $\F_2$: the conjugacy class of $[e_1,e_2]$ 
has exactly two images under $Aut(\F_2)$ (see \cite[Proposition 5.1,p.44]{LyndonSchupp}).

\begin{theorem}
Let $G$ be a torsion-free hyperbolic group. Then any definable proper subgroup of $G$ is cyclic. 
\end{theorem}
\begin{proof}
We may assume that $G$ is not cyclic. Suppose the result is not true. We first consider the case where $G$ is elementarily equivalent to 
a free group. Then there is a definable (over $\F_2$) non abelian subgroup of $\F_2$, which we denote by $A$. Let $E_A$ 
be the (definable over $\F_2$) equivalence relation defined by $E_A(a,b)$ if and only if $a\cdot A = b\cdot A$. Let $R_A$ 
be the relation given by Theorem \ref{Elim}. Let $R_A(e_3,\F_3^{eq})=\{\bar{b}_1,\ldots,\bar{b}_k\}$. 
We claim that there is $i\leq k$ such that $\bar{b}_i$ is in $\F_3^{eq}\setminus \F_2^{eq}$. Otherwise, since $\F_2$ is an elementary substructure, 
we find some $c\in \F_2$ with $R_A(c,\F_3^{eq})=\{\bar{b}_1,\ldots,\bar{b}_k\}$, so $c$ and $e_3$ are equivalent. 
But  then $e_3\cdot c^{-1}$ is in $A$. Since $A$ is definable over $\F_2$ we have that $A$ is generic, contradicting the 
connectedness of $T_{fg}$. So, without loss of generality, $\bar{b}_1\in \F_3^{eq}\setminus\F_2^{eq}$.
Now consider the automorphisms that fix $\F_2$ and send $e_3$ to $e_3\cdot a$ for some $a\in A$: 
every such automorphism fixes $R_A$, and  clearly $E_A(e_3,e_3\cdot a)$. Since $A$ is not abelian however, 
we can find $a, c$ such that $[a,c]\neq 1$: by Corollary \ref{corMalOrb}, $\bar{b}_1$ has 
infinitely many images under the iterates of $f_a$ and $f_c$, a contradiction.

Finally, suppose $G$ is not elementarily equivalent to a free group. Is not hard to see that the above argument is still valid 
if we replace $\F_2$ with the elementary core of $G$ and use Theorem \ref{HypGen}.

\end{proof}

The following corollary follows immediately from the fact that the normalizer of a cyclic 
subgroup in a torsion-free hyperbolic group coincides with its (cyclic) centralizer:

\begin{corollary}\label{DS}
Let $G$ be a non-cyclic torsion-free hyperbolic group. Then $Th(G)$ is definably simple.
\end{corollary}

We give another application of the weak elimination of imaginaries extending a result of 
Kharlampovich-Myasnikov (see \cite[Corollary 2]{MalcevKM}) that proved that the set of primitive elements of $\F_n$ is not definable if $n>2$.

In \cite[Theorem 4.8]{TowersLPS} it was proved that no (non-trivial) type in $S(T_{fg})$ is isolated, equivalently since free groups are homogeneous structures 
no finite tuple has $\emptyset$-definable orbit under $Aut(\F_n)$. An easy application of the weak elimination of imaginaries 
allows us, in free groups of rank at least $3$, to strengthen this 
result by showing that no orbit is definable even with parameters.
Intuitively the reason is that
every such orbit is $Aut(\F_n)$-invariant but every non trivial canonical parameter can be 
``moved'' by an automorphism of $\F_n$, thus the orbit can only be definable over $\emptyset$, contradicting 
the above mentioned theorem. Note that it is well known that the orbit of any tuple in $\F_2$ (under $Aut(\F_2)$) is 
definable in $\F_2$ by a result of Nielsen \cite{Nielsen}.

\begin{proposition}
Let $\bar{v}$ be a non trivial tuple of elements in a non abelian free group $\F$ of rank at least $3$. Then $Aut(\F).\bar{v}$ 
is not definable.
\end{proposition}
\begin{proof}
Suppose it is, and let $\mathfrak{e}=[\bar{a}]_E$ be the canonical parameter for the definable set $X:= Aut(\F).\bar{v}$, in $\F^{eq}$. 
By Theorem \ref{0Elim}, $\abs{R_E(\bar{a},\F^{eq})}$ is finite and by Theorem 4.8 in \cite{TowersLPS}, it contains a non trivial tuple, 
thus by Lemma \ref{InfOrb} there is an automorphism such that $f(\mathfrak{e})\neq \mathfrak{e}$, but then $f(X)\neq X$, a contradiction.  
 \end{proof}

\section{Abelian interpretable groups in torsion-free hyperbolic groups}\label{AbIntGr}

We show that if a formula $\phi$ over $\F_n^{eq}$ ``gains'' an element in $\F_{n+1}^{eq}$, 
i.e. $\phi(\F_n^{eq})\neq\phi(\F_{n+1}^{eq})$, then it cannot be given definably the structure of an abelian group. 
In the case of a torsion-free hyperbolic group not elementarily equivalent to a free group, we prove an analogous result 
for the elementary core.

\begin{lemma}\label{conj}
Let $G := F_0 * X_0 * X_1$ be a free product of groups and let $h$ be an automorphism of $G$ fixing $F_0$ pointwise and acting
on $X_0*X_1$ as an automorphism of prime order $p$ whose fixed point set is exactly $X_0$.

Suppose that $h$ fixes the conjugacy class of a cyclically reduced element $a\in G\setminus (F_0*X_0 \cup X_0*X_1)$, of the form
\[a=a_1x_1a_2x_2\cdots a_mx_m \mbox{ with }a_i\in F_0 \setminus \{1\} \textrm{ and } x_i\in X_0*X_1 \setminus \{1\} \textrm{ for } i=1,\ldots m.\]

Then there is a permutation $\sigma\in \langle (1.....m)\rangle\leq Sym(\{1,\ldots m\})$
such that:
\begin{enumerate}
\item[(i)] $a_i=a_{\sigma(i)}$ and $h(x_i)=x_{\sigma(i)}$;
\item[(ii)] $p$ divides the order $o(\sigma)$ of $\sigma$.
\end{enumerate}
\end{lemma}

\begin{proof} Since $h$ is the identity on $F_0$ and leaves $X_0*X_1$ invariant and because
$a$ is cyclically reduced and conjugate to $h(a)$,
we see that \[h(a)=a_1h(x_1)a_2h(x_2)\ldots a_mh(x_m)\] differs from
$a$ by a cyclic shift only. This implies $(i)$.

Now $(ii)$ follows from $(i)$: we see that $h^{o(\s)}$ fixes each of the $x_i$. 
The elements of $X_0*X_1$ fixed by $h$ (and thus by any element generating $\langle h \rangle$) 
are exactly those of $X_0$. Thus $h^{o(\s)}$ must be the identity so $o(h)=p$ divides $o(\s)$.
\end{proof}

We will need the following small lemma:

\begin{lemma}\label{cyclic}
Let $M$ be a finite cyclic group, let $s,t\in M$ and let $d>1$.
Then there are $k,l$ 
with $s^k=t^l$ and such that $d$ does not divide both $k$ and $l$.
\end{lemma}
\begin{proof}
Write $o(s) = ef$, and $o(t) = eg$ where $f, g$ are coprime. Then $s^f$ and $t^g$ have order $e$, hence generate the same subgroup of $M$. So each can be written as a power of the other and the conclusion follows.
\end{proof}

\begin{lemma}\label{conj2}

Let $G := U_f * W * U_g$ be a free product of groups. Let $f$ and $g$  be automorphisms of $G$ of prime order $p$ such that $f$ 
is the identity on $W$ and on $U_f$ and acts without nontrivial fixed point on $U_g$, and $g$ is the identity on 
$W$ and on $U_g$ and acts without nontrivial fixed point on $U_f$.
If $a$ is an element of $G$ whose conjugacy class is fixed both by $f$ and by $g$, then $a$ lies in 
$W*U_f$ or in $W*U_g$.
\end{lemma}

\begin{proof} 
Suppose that $a\notin (W*U_f)\cup (W*U_g)$.
We may choose $a=a_1x_1a_2x_2\cdots a_mx_m$ cyclically reduced with $a_i\in U_f*W , x_i\in U_g$, and $ a_i\neq 1\neq x_i$ for $i=1,\ldots m$. 
By Lemma~\ref{conj} applied to $f$ and the decomposition $G= (U_f * W) * 1 * U_g$ on the one hand, and to $g$ and the 
decomposition $G = U_g * W * U_f$ on the other hand
we find $\s,\tau \in\langle(1\ldots m)\rangle$
such that for $i=1,\ldots m$ we have

\[a_i=a_{\s(i)} \mbox{ and\ } f(x_i)=x_{\s(i)}\]
\[g(a_i)=a_{\tau(i)} \mbox{ and\ } x_i=x_{\tau(i)}.\]

By Lemma~\ref{cyclic} we can choose $k,l$ 
with $\s^k=\tau^l$ such that $p$ does not divide both $k$ and $l$.

For $i=1,\ldots m$ we now have
\[f^k(x_i) = x_{\s^k(i)} = x_{\tau^l(i)} = x_i,\]
\[g^l(a_i) = a_{\tau^l(i)} = a_{\s^k(i)} = a_i.\]
Thus $f^k$ is the identity (and so $p$ divides $k$)
and $g^l$ fixes each $a_i$. Since $p$ does not divide $l$,
we conclude that $a_i\in W$ for $i=1,\ldots m$ showing that $a\in W*U_g$. This
contradiction proves the lemma.
\end{proof}

\begin{lemma}\label{basic}
Let $H:=G*C$ be a torsion-free hyperbolic group. Let $f$ be an automorphism of $H$ of 
prime order $p>2$ that fixes $G$ and acts on $C$ without non-trivial fixed point. 

Suppose $\beta\in H^{we}\setminus G^{we}$ is not a conjugacy class, i.e. $\beta$ is not in $S_1(H)$. 
Then $|\{\beta,f(\beta),$ $f^2(\beta),\ldots,f^{p-1}(\beta)\}|=p$.  
\end{lemma}

\begin{proof}
We first note that no non-trivial power of $f$ fixes an element in $H\setminus G$. This also implies that $f^i(d)\neq d^{-1}$ for all $d \in H\setminus G$, otherwise we would have $f^{2i}(d) = d$. We now take cases according to the list of basic sorts.

We will only treat the case where $\beta=[(b_1,b_2,b_3)]_{E_{4_{k,l}}}$. Suppose for the sake of contradiction that 
$(f^i(b_1),f^i(b_2),f^i(b_3))\sim_{E_{4_{k,l}}}(f^j(b_1),f^j(b_2),f^j(b_3))$ for $i\not\equiv j \mod p$. So we have $[f^i(b_1),f^j(b_1)]=1$ and 
$[f^i(b_3),f^j(b_3)]=1$. Since $H$ is a torsion-free hyperbolic group we have  $f^{i-j}(b_1)=b_1$ or $f^{i-j}(b_1)=b_1^{-1}$. But since
$f^{i-j}$ cannot invert or fix an element in $H\setminus G$, it follows that $b_1\in G$ and  similarly $b_3\in G$. 
Now since $\beta\in H^{we}\setminus G^{we}$, we must have $b_2\in H\setminus G$. But then
$\gamma f^i(b_2)\epsilon = f^j(b_2)$ for some $\gamma,\epsilon\in G$, and an easy calculation 
shows that this is not possible.

\end{proof}

We will apply  Lemma \ref{basic} to the group  $H:=G*\F_p$ where $G$ is torsion-free hyperbolic, 
$p>2$ is a prime and $f$ is the automorphism of $H$ that fixes $G$ and cyclically permutes $(e_1,\ldots,e_p)$.

\begin{proposition}\label{LT}
Let $X$ be a definable set in $\F_n^{eq}$. Suppose $X(\F_{n+1}^{eq})\neq X(\F_n^{eq})$. Then
$X$ cannot be given definably the structure of an abelian group.
\end{proposition}
\begin{proof}
Suppose otherwise and let $(X,\odot)$ be an abelian group. Suppose
$X$ is a subset of sort $S_E$ (for some $\emptyset$-definable equivalence relation $E$). By Theorem \ref{Elim} we can assign 
to $E$ a definable equivalence relation $R_E$ such that $R_E(\bar{a},\bar{y})$ cannot have more than $s$ solutions 
(for any $\bar{a}$) and each solution is a tuple containing $l$-many elements in imaginary sorts.

Let $p$ be a prime greater than $max\{s,2\}$. Let $A_i:=\langle e_{n+ip+1},\ldots, e_{n+(i+1)p}\rangle$ for $i\leq l\cdot s$ 
and $f_i$ be the automorphism of $\F_n*A_i$ which is the identity on $\F_n$ and cyclically permutes the given basis of $A_i$.

Now let $\beta_0\in X(\F_{n+1}^{eq})\setminus X(\F_n^{eq})$ and $\beta_i$ be the image of $\beta_0$ under 
the automorphism of $\F_{\omega}:=\langle e_1,e_2,\ldots,e_n,\ldots\rangle$ 
which exchanges $e_{n+1}$ with $e_{n+ip+1}$ and fixes all other elements of the basis. 

For each $0\leq i\leq s\cdot l$ we consider the following product of elements of $X(\F_{\omega}^{eq})$.
$$ \beta_i\odot f_i(\beta_i)\odot\ldots\odot f_i^{p-1}(\beta_i)=\gamma_i$$
Is not hard to see that $\gamma_i$ is an element of $(\F_n*A_i)^{eq}\setminus \F_n^{eq}$. Since $(X,\odot)$ is abelian 
we have $f_i(\gamma_i)=\gamma_i$. 

Finally we consider the product 

$$\gamma_0\odot\gamma_1\odot\ldots\odot\gamma_{l\cdot s}=\epsilon$$ 

Let $G:=\F_n*A_0*A_1*\ldots*A_{ls}$ and $G_j:=\F_n*B_j$, where $B_j$ is a free product 
with strictly less than $l\cdot s+1$ distinct free factors from the list $\{A_0,\ldots,A_{l\cdot s}\}$. Then is not hard to see that 
$\epsilon$ is an element in $G^{eq}\setminus \bigcup_j G_j^{eq}$.  
Note that since all the $\gamma_i$ are fixed by each $f_j$ (or rather by the obvious extension of $f_j$ to $\F_{\omega}$), 
we have $\epsilon$ is fixed by each $f_j$. 

Let $\epsilon=[\bar{e}]_E$, and suppose that $R_E(\bar{e},G^{eq})$ contains an element (in a tuple) which 
is not a conjugacy class and lives in $G^{eq} \setminus \F_n^{eq}$. Then by Lemma \ref{basic} this element  
has $p$ distinct images by the powers of some $f_i$, a contradiction. Thus, we may assume that all elements 
in the tuples of the solution set $R_E(\bar{e},G^{eq})$ that live in $G^{eq}\setminus \F_n^{eq}$ are 
conjugacy classes. By repeatedly applying Lemma \ref{conj2}, we see that each 
conjugacy class is an element in $(\F_n*A_i)^{eq}$ for some $0\leq i\leq l\cdot s$, but since there are less 
than $l\cdot s$ conjugacy classes this contradicts the fact that $\epsilon$ is an element in 
$G^{eq}\setminus \bigcup_j G_j^{eq}$.
\end{proof}

\begin{proposition}\label{HLT}
Let $G$ be a torsion-free hyperbolic group not elementarily equivalent to a free group. Let $X$ be a set definable 
in $EC(G)^{eq}$ such that $X(EC(G)^{eq})\neq X((EC(G)*\mathbb{Z})^{eq})$. Then $X$ cannot be given definably 
the structure of an abelian group.
\end{proposition}

\begin{proof}
The proof is identical to the proof of Proposition \ref{LT} replacing $\F_n$ by $EC(G)$ and using Proposition~\ref{HypChain}.
\end{proof}

\section{The generic type is foreign to any interpretable abelian group}\label{Generic}

In this section we bring everything together in order to prove Theorem \ref{For}. Before we start we give a brief account of 
the stability theoretic tools we use. For a more thorough exposition the reader is referred to \cite{PillayStability}. 
Let us fix a complete stable theory $T$, countable if you wish, 
and a very saturated model $\mathbb{M}$ of $T$. $\mathcal{M},\mathcal{N},..$ denote small elementary submodels, and $A,B,..$ small subsets.  
We repeat a definition and fact from Chapter 7 of \cite{PillayStability}. See Definition 7.4.1 and 7.4.7 there (originally due to Hrushovski). 

\begin{definition} Let $p(x)\in S(A)$ be a stationary type and $\Sigma(y)$ a partial type over some small set $B$ of parameters.
\newline
(i) We say that $p$ is foreign to $\Sigma$ if for any model $\mathcal{M}$ containing $A\cup B$, any realization $a$ of $p|\mathcal{M}$ 
(the nonforking extension of $p$ over $\mathcal{M}$) and any realization $b$ of $\Sigma$, $a$ is independent from $b$ over $\mathcal{M}$.
\newline
(ii) We say that $p$ is internal to $\Sigma$ if for some $\mathcal{M}$ containing $A\cup B$, and some realization $a$ of 
$p|\mathcal{M}$ there is a tuple $c$ of realizations of $\Sigma$ such that $a\in dcl(\mathcal{M},c)$.
\end{definition} 

\begin{fact}\label{Inter}
Suppose $G$ is a connected definable group, defined over a set $A$, and let $p(x)\in S(A)$ be the generic type of $G$. 
Suppose that $\Sigma(y)$ is a partial type over some small set $B$ of parameters and that $p$ is not foreign to $\Sigma$. 
Then there is a normal $A\cup B$-definable subgroup $N$ of $G$ such that the generic type (over $A\cup B$) of $G/N$ is internal to $\Sigma$.
\end{fact}  

We now specialize to torsion-free hyperbolic groups. Let $G$ be a non-cyclic torsion-free hyperbolic group and 
$\mathcal{M}\models Th(G)$. We write $p^G_{0}|\mathcal{M}$ for the unique nonforking extension of $p^G_{0}$ over $\mathcal{M}$. 

\begin{theorem}\label{InterAb}
The generic type of a non-cyclic torsion-free hyperbolic group is foreign to any interpretable abelian group.
\end{theorem}

\begin{proof} Suppose $p^G_0$ is not foreign to some interpretable abelian group $A$. By Fact \ref{Inter} and Corollary \ref{DS}, 
there is a model $\mathcal{M}$ of $Th(G)$ (over which $A$ is defined), and a realization $b$ of $p^G_{0}|\mathcal{M}$ and a tuple $(c_{1},..,c_{n})$ 
of elements of $A$, such that $b\in dcl(\mathcal{M}, c_{1},..,c_{n})$. We may assume that 
$\mathcal{M}$ is an elementary extension of the model $\F_{2}$ (respectively $EC(G)$ 
in the case where $G\not\models T_{fg}$). Suppose  $b = f(c_{1},..,c_{n},m)$ where 
$f(-)$ is a partial $\emptyset$-definable function, and $m$ is a tuple from $\mathcal{M}$. We may assume $A$ is defined over $m$ too, 
by formula $\psi(y,m)$  (where of course $y$ is a variable from the appropriate imaginary sort). Then 
the formula  $\theta(x,m)$:  $``\psi(y,m)$ defines an abelian group" 
$\wedge$  ($\exists y_{1},..,y_{n}((\wedge_{i}\psi(y_{i},m)) \wedge x = f(y_{1},..,y_{n},m))$ is in $p^G_0|\mathcal{M}$. 
As $p^G_0|\mathcal{M}$ is definable over $\emptyset$, we can find $m'\in \F_{2}$ (respectively $EC(G)$) such that
$\theta(x,m')\in p^G_0|\F_2$ (respectively $p^G_0|EC(G)$), so is satisfied by $e$ in $\F_3:=\F_2*\langle e\rangle$ 
(respectively $EC(G)*\langle e\rangle$).  
The formula $\psi(y,m')$ then defines an abelian group, $B$ say, and there are elements $d_{1},..,d_{n}\in B$ 
such that $e = f(d_{1},...,d_{n},m')$. As $\F_2*\langle e\rangle$ (respectively $EC(G)*\langle e\rangle$) is a 
model containing $\F_{2}$ (respectively $EC(G)$) and $e$ we can find such
$d_{1},..,d_{n}\in B(\F_3^{eq})$ (respectively $(EC(G)*\Z)^{eq}$).  Hence  $e$ and $(d_{1},..,d_{n})$ are 
interdefinable over $\F_{2}$ (respectively $EC(G)$). 
So, for some $i\leq n$ we have $d_i\in B(\F_3^{eq})\setminus \F_2^{eq}$ (respectively $B((EC(G)*\Z)^{eq})\setminus EC(G)^{eq}$), 
contradicting Proposition \ref{LT} (respectively Proposition \ref{HLT}).
\end{proof}

\paragraph{Acknowledgements.} We wish to thank A. Ould Houcine for pointing out that our results on 
interpretable abelian groups and definable simplicity extend easily from 
non-abelian free groups to (non-cyclic) torsion-free hyperbolic groups. 
The third named author would like to thank Frank Wagner for some stimulating questions.

\providecommand{\bysame}{\leavevmode\hbox to3em{\hrulefill}\thinspace}
\providecommand{\MR}{\relax\ifhmode\unskip\space\fi MR }
\providecommand{\MRhref}[2]{%
  \href{http://www.ams.org/mathscinet-getitem?mr=#1}{#2}
}
\providecommand{\href}[2]{#2}


\begin{thebibliography}{OHT12}

\bibitem[EPP90]{1BEPP}
David Evans, Anand Pillay, and Bruno Poizat, \emph{A group in a group}, Algebra
  and Logic \textbf{29} (1990), 244--252.

\bibitem[KM11]{MalcevKM}
Olga Kharlampovich and Alexei Myasnikov, \emph{Definable sets in a hyperbolic
  group}, arXiv:1111.0577v4 [math.GR], 2011.

\bibitem[LPS11]{TowersLPS}
Lars Louder, Chlo\'e Perin, and Rizos Sklinos, \emph{Hyperbolic towers and
  independent generic sets in the theory of free groups}, to appear in the
  Proceedings of the conference "Recent developments in Model Theory", Notre
  Dame Journal of Formal Logic, 2012.

\bibitem[LS77]{LyndonSchupp}
R.C. Lyndon and P.E. Schupp, \emph{Combinatorial group theory},
  Springer-Verlag, 1977.

\bibitem[Nie17]{Nielsen}
Jacob Nielsen, \emph{{Die Isomorphismen der allgemeinen unendliehen Gruppe mit
  zwei Erzeugenden}}, Mathematische Annalen (1917), 385--397.

\bibitem[OH11]{OuldHoucineHomogeneity}
A. Ould~Houcine, \emph{Homogeneity and prime models in torsion-free
  hyperbolic groups}, Confluentes Mathematici \textbf{3} (2011), 121--155.

\bibitem[OHT12]{AmpleOK}
A.~Ould~Houcine and K.~Tent, \emph{Ampleness in the free group},
  arXiv:1205.0929v2 [math.GR], 2012.

\bibitem[Pil96]{PillayStability}
Anand Pillay, \emph{Geometric stability theory}, Oxford University Press, 1996.

\bibitem[Pil08]{PillayForking}
\bysame, \emph{Forking in the free group}, J. Inst. Math. Jussieu \textbf{7}
  (2008), 375--389.

\bibitem[Pil09]{PillayGenericity}
\bysame, \emph{On genericity and weight in the free group}, Proc. Amer. Math.
  Soc. \textbf{137} (2009), 3911--3917.

\bibitem[Poi83]{PoizatGenericAndRegular}
Bruno Poizat, \emph{Groupes stables, avec types g\'en\'eriques r\'eguliers}, J.
  Symbolic Logic \textbf{48} (1983), 339--355.

\bibitem[PS09]{PerinSklinosHomogeneity}
Chlo\'e Perin and Rizos Sklinos, \emph{Homogeneity in the free group}, 
Duke Math. J. \textbf{161} (2012), 2635--2668. 


\bibitem[Sela]{SelaIm}
Zlil Sela, \emph{{Diophantine geometry over groups IX: Envelopes and
  Imaginaries}}, {preprint, available at
  \url{http://www.ma.huji.ac.il/~zlil/}}.

\bibitem[Selb]{SelaStability}
\bysame, \emph{{Diophantine geometry over groups VIII: Stability}}, {to appear
  in the Ann. of Math. (2), available at \url{http://www.ma.huji.ac.il/~zlil/}}.

\bibitem[Sel09]{Sel7}
\bysame, \emph{Diophantine geometry over groups. {VII}. {T}he elementary theory
  of a hyperbolic group}, Proc. Lond. Math. Soc. (3) \textbf{99} (2009), 217--273.

\bibitem[Skl11a]{SklinosGenericType}
Rizos Sklinos, \emph{On the generic type of the free group}, J. of Symbolic
  Logic \textbf{76} (2011), 227--234.

\bibitem[Skl11b]{ThesisSklinos}
\bysame, \emph{Some model theory of the free group}, Ph.D. thesis, University
  of Leeds, 2011.

\bibitem[Skl12]{AmpleSk}
\bysame, \emph{A note on ampleness in the theory of non abelian free groups},
  arXiv:1205.4662v2 [math.LO], 2012.

\end{thebibliography}
\end{document}